\documentclass[10pt]{article}

\usepackage{graphicx}
\usepackage{epstopdf}
\usepackage[english]{babel}
\usepackage[latin1]{inputenc}

\usepackage{color}
\usepackage{amsmath,
 amsfonts, latexsym,url}
\usepackage{amssymb}
\usepackage{cite}

\usepackage[labelfont=bf,labelsep=period,justification=raggedright]{caption}

\makeatletter
\renewcommand{\@biblabel}[1]{\quad#1.}
\makeatother

\date{}

\newtheorem{notation}{Notation}
\newtheorem{definition}{Definition}
\newtheorem{theorem}{Theorem}
\newtheorem{corollary}{Corollary}
\newtheorem{remark}{Remark}
\newtheorem{lemma}{Lemma}
\newtheorem{algorithm}{Algorithm}
\newtheorem{proof}{Proof}

\begin{document}

\begin{flushleft}
{\Large
\textbf{Convergence Time Towards Periodic Orbits in Discrete Dynamical Systems}
}
\\
Jes\'us San Mart\'in$^{1,\dagger}$, 
Mason A. Porter$^{2,\ast}$
\\
\bf{1} Escuela T\'ecnica Superior de Ingenier\'ia y Dise\~{n}o Industrial (ETSIDI), Universidad Polit\'ecnica de Madrid, Ronda de Valencia 3, Madrid, Spain
\\
\bf{2} Oxford Centre for Industrial and Applied Mathematics, Mathematical Institute, University of Oxford, Oxford, UK
\\
$\dagger$ E-mail: jesus.sanmartin@upm.es
\\
$\ast$ E-mail: porterm@maths.ox.ac.uk
\end{flushleft}


\section*{Abstract}

We investigate the convergence towards periodic orbits in discrete dynamical systems. We examine the probability that a randomly chosen point converges to a particular neighborhood of a periodic orbit in a fixed number of iterations, and we use linearized equations to examine the evolution near that neighborhood.  The underlying idea is that points of stable periodic orbit are associated with intervals. We state and prove a theorem that details what regions of phase space are mapped into these intervals (once they are known) and how many iterations are required to get there.  We also construct algorithms that allow our theoretical results to be implemented successfully in practice.







\section*{Introduction}\label{intro}

Periodic orbits are the most basic oscillations of nonlinear systems, and they also underlie extraordinarily complicated recurrent dynamics such as chaos
 \cite{scholarpedPO,chaosbook,Poincare1892,Auerbach1987,Artuso1990}.  Moreover, they occur ubiquitously in applications throughout the sciences and engineering.  It is thus important to develop a deep understanding of periodic dynamics.

It is important and common to question how long it takes a point in phase space to reach a stable periodic orbit from an arbitrary initial condition.  When studying synchronization and other forms of collective behavior, it is crucial to examine not only the existence of stable periodic orbits but also the time that it takes to converge to such dynamics in both natural and human-designed systems \cite{Ermentrout2010,Strogatz2000,Strogatz1994}.  For example, it is desirable to know how long it will take an engineered system that starts from an arbitrary initial condition to achieve the regular motion at which it is designed to work \cite{Lellis2013,Yu2012}.  A system can also be perturbed from regular motion by accident, and it is important to estimate how long it will take to return to regular dynamics.  Similar questions arise in physics \cite{Valtaoja2000,Kreilos2012}, biology \cite{Neufeld2012,Ermentrout2010,cowsync}, and many other areas. It is also important to consider the time to synchronize networks \cite{qi08,grabow10,nishikawa10} and to examine the convergence properties of algorithms for finding periodic orbits \cite{kleb01,chaosbook}.
 
To study the problem of convergence time to periodic orbits, let's first consider the Hartman-Grobman Theorem \cite{Grobman1959,Hartman1960}, which states that the flow of a dynamical system (i.e., a vector field) near a hyperbolic equilibrium point is topologically equivalent to the flow of its linearization near this equilibrium point.  If all of the eigenvalues of the Jacobian matrix evaluated at an equilibrium have negative real parts, then this equilibrium point is reached exponentially fast when one is in a small neighborhood of it.  To determine convergence time to a hyperbolic equilibrium, we thus need to calculate how long it takes to reach a neighborhood of the equilibrium from an arbitrary initial condition.  After reaching the neighborhood, the temporal evolution is then governed by a linear dynamical system (which can be solved in closed form).  An analogous result holds for hyperbolic periodic orbits in vector fields \cite{Lan13}.  
To turn periodic orbits in vector fields into fixed points in maps, one can use Poincar\'e return maps, which faithfully capture properties of periodic orbits.  A Poincar\'e  map can be interpreted as a discrete dynamical system, so the problem of determining how long it takes to reach a hyperbolic stable periodic orbit from arbitrary initial conditions in a vector field is reduced to the problem of determining how long it takes to reach the neighborhood of a hyperbolic fixed point in a discrete dynamical system.

Our work considers how long it takes to reach a periodic orbit of a differential equation---starting from an arbitrary point in phase space---by using a Poincar\'e return map of its associated vector field.  For simplicity, suppose that a return map (which is built from a Poincar\'e section) is unimodal. If we approximate the unimodal Poincar\'e map by using a unimodal function $f(x)$, then we can use $f(x)$ in our algorithm to estimate the convergence time to the periodic orbit.  Periodic motion is ubiquitous in models (and in nature), and it is important to explore how long it takes to converge to such behavior.

In this paper, we prove a theorem for the rate of convergence to stable periodic orbits in discrete dynamical systems.  Our basic strategy is as follows. We define the neighborhood $I_p$ of  a hyperbolic fixed point, and we calculate what fraction $\omega$ of the entire phase space $I$ is mapped into $I_p$ after $q$ iterations.  Using $\mu(w)$ and $\mu(I)$, respectively, to denote the measures of $w$ and $I$, a point that is selected uniformly at random from $I$ has a probability of ${\mu(w)}/{\mu(I)}$ to reach $I_p$ in $q$ iterations. To illustrate our ideas, we will work with a one-dimensional (1D) discrete dynamical system $x_{n+1} = f(x_n;r)$ that is governed by a unimodal function $f$ and is parametrized by a real number $r$. We focus on unimodal functions for two primary reasons: (i) many important results in dynamical systems are based on such functions; and (ii) it is simpler to illustrate the salient ideas using them than with more complicated functions.

To determine the set that is mapped into $I_p$, we take advantage of the fact that points in periodic orbits are repeated periodically, so their corresponding neighborhoods must also repeat periodically.  In theory, an alternative procedure would be to iterate backwards from $I_p$, but this does not work because one cannot control successive iterations of $f^{-1}$.  The function $f$ is unimodal, so it is not bijective and in general one obtains multiple sets for each backward iteration of a single set. The number of sets grows geometrically, and one cannot in general locate them because an analytical expression for $f^{-1}$ is not usually available. 

To explain the main ideas of this paper and for the sake of simplicity, consider a stable periodic orbit $O_p$ of period $p$ that is born in $p$ saddle-node bifurcations of $f^p$. Every point  $x_i$ (with $i \in \{1, \ldots ,p\})$ of $O_p$ has a sibling point $x_i^*$ that is born in the same saddle-node bifurcation. Because $f^p(x_i)=x_i$ and $f^p(x_{i}^{*})= x_{i}^{*}$, it follows that $f^p(I_i)=I_i$, where $I_i=[ x_i, x_{i}^{*}]$. That is, $x_i$, $x_{i}^{*}$, and $I_i$ all repeat periodically. Roughly speaking, we will build the interval $I_p$ from the interval $I_i$.

Consider a plot in which points along the horizontal axis are mapped via $f$ to points along the vertical axis (as is usual for 1D maps). The orbit $O_p$ is periodic with period $p$, so $x_j \in O_p$ implies that $\{(x_j, f^q(x_j)),\; q=0,1,\ldots \}$ yields $p$ periodic points with a horizontal axis location of $x_j$.  We say that these points are located in the ``column" $x_j$. Because $f^q(x_j )=x_i\in O_p$ for some $q$, we obtain $p$ points located in the same column $x_j$. These points are given by $\{(x_j, f^q(x_j))=(x_j,x_i),\; i=1,\ldots , p  \}$. As we have indicated above, each point $(x_j,x_i)$ is associated with an interval $I_i$.  No matter how many iterations we do, the fact that the orbit is periodic guarantees that there are exactly $p$ intervals in the same place (where the points $(x_j,x_i)$ are located).   We thereby know the exact number and locations of all intervals.

To complete the picture, we must also take into account that if there exists an interval $W_{q_{ij}}$ such that $f^q(W_{q_{ij}})=I_i$, then any point of $W_{q_{ij}}$ will reach a point of $I_i$ in at most $q$ iterations. The geometric construction above yields the interval $W_{q_{ij}}$, as one can see by  drawing a pair of parallel line segments that intersect both $f^q$ and the endpoints of the interval $I_i$.  We will approximate $f^q$ by a set of such line segments so that we can easily calculate the intersection points.

The remainder of this paper is organized as follows. First, we give definitions and their motivation. We then prove theorems that indicate how long it takes to reach the interval $I_i$ from an arbitrary initial condition. We then construct algorithms to implement the results of the theorems. Finally, we discuss a numerical example and then conclude.


\section*{Definitions}\label{def}

Consider the discrete dynamical system
\begin{equation}
	x_{n+1}=f(x_{n};r)\,,\qquad f:I\rightarrow I\,, \qquad I=[a,b]\,, \label{eq-1}
\end{equation}
where $f(x;r)$ is a one-parameter family of unimodal functions with negative Schwarzian derivative and a critical point at $x = C$. Without loss of generality, we suppose that there is a (both local and global) maximum at $C$.  At a critical point of a map $f$, either $f^\prime=0$ (as in the logistic map) or $f^{\prime}$ does not exist (as in the tent map). Some of the results of this paper related with critical points only require continuous functions, which is a much weaker condition than the requirement of a negative Schwarzian derivative.

\begin{remark}\label{rem:schw_der}
Because $f$ has a negative Schwarzian derivative, $f^q$ does as well (because it is a composition of functions with negative Schwarzian derivatives). By using the chain rule, we obtain $f^{q\prime}=0$ only at extrema. Therefore,  $f^{q\prime}\neq 0$ between consecutive extrema. The Minimum Principle \cite{Brin2002} for a function with negative Schwarzian derivative then guarantees that there is only one point of inflection between two consecutive extrema of $f^q$. If there were more than one point of inflection, then $f^{q\prime\prime}=0$ at least two points. One of them would be a maximum of $f^{q\prime}$, and the other one would be a minimum. This contradicts the Minimum Principle. Consequently, the graph of $f^q$ between two consecutive extrema has a sigmoidal shape (i.e., it looks like $\left\lmoustache{}\right.$ or $\left.{}\right\rmoustache$), which becomes increasingly steep as $q$ becomes larger.  
This fact makes it possible to approximate $f^q$ between two consecutive extrema by a line segment near the only point of inflection that is located between two consecutive extrema.
\end{remark}

Because the Schwarzian derivative of $f$ is negative, Singer's Theorem  \cite{Singer1978} ensures that the system (\ref{eq-1}) has no more than one stable orbit for every fixed value of the parameter $r$. Additionally, the system (\ref{eq-1}) exhibits the well-known Feigenbaum cascade \cite{Myrberg63,Feigenbaum78,Feigenbaum79}, which we show in the bifurcation diagram in Fig.~\ref{fig:diagrama}. 

For a particular value of the parameter $r$, the map $f^p$ has $p$ simultaneous saddle-node (SN) bifurcations, which result in an SN $p$-periodic orbit. As $r$ is varied, the SN orbit bifurcates into a stable orbit $\{ S_i \}_{i=1}^p$ and an unstable orbit $\{ U_i \}_{i=1}^p$. The points $S_i$ and $U_i$ are, respectively, the node and the saddle generated in an SN bifurcation, so $U_i$ is the nearest unstable point to $S_i$ (see Fig.~\ref{fig:figura1b}). In other words, the points in the stable orbits (called ``node orbits") are node points, whereas the points in the unstable orbit (called ``saddle orbits") are saddle points. From Remark \ref{rem:schw_der}, we know that the neighborhoods of these points are concave or convex.

The derivative of $f^p$ is 1 at the fixed point where the SN bifurcation takes place. As one varies $r$, the derivative evaluated at that bifurcation point changes continuously from $1$ to $-1$.  When the derivative is $-1$, the stable orbit (i.e., the node orbit) undergoes a period-doubling bifurcation. As a result, the stable orbit becomes unstable (yielding the orbit $\{ U_i \}_{i=1}^p$) and two new stable orbits ($\{ S_{i_{1}} \}_{i_1=1}^p$ and $\{ S_{i_{2}} \}_{i_2=1}^p$) appear. The points $S_{i_{1}}$ and $S_{i_{2}}$ are nodes, and the point $U_i$ is a saddle. From our geometric approach, the intervals $(S_{i_{1}}, U_i)$ and $(S_{i_{2}}, U_i)$ that are generated via the period-doubling bifurcation behave in the same way as the interval $(S_i, U_i)$ that was generated in the SN bifurcation. Therefore, we can drop the indices ``1" and ``2" and write $(S_i, U_i)$ for the orbits that arise from both the SN bifurcation and the period-doubling bifurcation.


\begin{notation}\label{not:not-U} Let $U_i^{\prime}$ denote the nearest point to $S_i$ that results from the intersection of the line $x_{n+1}=U_i$ with $f^p$ (see Figs.~\ref{fig:figura1b}, \ref{fig:figura1a}, and \ref{fig:figura1c}).
\end{notation}

\begin{definition}\label{def:cap-inter}
Consider the points $x_i$ and $x_i^{\prime}$ that satisfy $f^p(x_i)=U_i$ and $f^p(x_i^{\prime})=U_i^{\prime}$. If $f^p$ is concave (respectively, convex) in a neighborhood of $S_i$, we say that $I_{P_i}=(x_i,x_i^{\prime})$ [respectively, $I_{P_i}=(x_i^{\prime},x_i)$] is the {\tt $i$th capture interval} of the stable $p$-periodic orbit $\{ S_i  \}_{i=1}^p$ and that $\displaystyle I_P=\bigcup_i I_{P_i}$ is the {\tt (aggregate) capture interval} of the stable $p$-periodic orbit $\{ S_i  \}_{i=1}^p$.
\end{definition}

\vspace{.2cm}

\begin{notation}\label{not:not-IP} Let $I_{P_{i,C}}$ denote the subinterval $I_{P_i}$ that contains the critical point $C$.
\end{notation}

From Definition \ref{def:cap-inter}, we see for all $x\in I_{P_i}$ that $f^{np}(x) \in I_{P_i}$ and $f^{np}(x) \longrightarrow S_i$ as $n \longrightarrow \infty$. Iterations of points $x\in I_{P_i}$ are repelled from $U_i$ and $U_i^{\prime}$, and they are attracted to $S_i$. The system (\ref{eq-1}) is linearizable around the fixed points $S_i$ and $U_i$. (Observe that $f(U_i^{\prime})=U_i$, so we also have control over this point.) Consequently, the convergence of iterations of $x \in I_{P_i}$ to $S_i$ is governed by the eigenvalues of the Jacobian matrix $Df$.

Because we can control the evolution inside $I_{P_i}$, we can examine how long it takes to reach $I_{P_i}$ starting from an arbitrary point $x \in I$. As we will see below, to obtain this result, we need to discern which subintervals of $I$ are mapped by $f^q$ into $I_{P_i}$ for arbitrary $q$. The first step in this goal is to split the interval $I$ in which $f^q$ is defined into subintervals in which $f^q$ is monotonic.

\begin{definition}\label{def:monotonicity}
Let $A=\{q_2,q_3,\ldots,q_{k-1} | q_2<\ldots<q_{k-1} \}$ be the set of points at which $f^q$ has extrema. Let $B= \{ q_1=a, q_k=b\}$, and we recall that we are considering the interval $I=[a,b]$. We will call $P_{\mathrm{mon-r}}=A \cup B= \{ q_1,q_2,q_3,\ldots,q_{k-1},q_k | q_1<q_2<\ldots<q_{k-1}<q_k \}$ the {\tt partition of monotonicity} of $f^q$. We will call $I_{q_j}=[q_j,q_{j+1}]$ (where $j=1,\ldots,k-1$) the {\tt $j$th interval of monotonicity} of $f^q$.
\end{definition}

By construction, $\displaystyle I=[a,b]= \bigcup_j I_{q_j}$, and $f^q$ is monotonic in $I_{q_j}$.  As we will explain below, one can calculate intervals of monotonicity $I_{q_j}$ easily by using Lemmas \ref{lem:extrema} and \ref{lem:coordinates}.

Once we know the intervals in which $f^q$ is monotonic, it is easy to obtain subintervals of $I$ that are mapped by $f^q$ into $I_{P_i}$.


We proceed geometrically (see Figs.~\ref{fig:figura1b}, \ref{fig:figura1c}, and \ref{fig:figura1d}):
\begin{itemize}
\item[(i)]{draw parallel lines through the points $x_i^{\prime}$ and $ x_i$ (i.e., through the endpoints of $I_{P_i}$);} 
\item[(ii)]{obtain the points at which the lines intersect $f^q$;}
\item[(iii)]{calculate which points are mapped by $f^q$ into the intersection points of (ii), for which one uses the fact that $f^q$ is monotonic in $I_{q_j}=[q_j,q_{j+1}]$;}
\item[(iv)]{determine, using the points obtained in (iii), the interval that is mapped by $f^q$ into $I_{P_i}$.}
\end{itemize}

Using this geometric perspective, we make the following definitions.

\begin{definition}\label{def:W}
Let
\begin{equation}
	A_{ij}= 
	\left\{
		\begin{array}{c}
			(x_{i,q_j},x_{i,q_j}^{\prime})\,, \qquad \mbox{if} \quad x_{i,q_j}<x_{i,q_j}^{\prime}\\
			(x_{i,q_j}^{\prime},x_{i,q_j})\,, \qquad \mbox{if} \quad x_{i,q_j}>x_{i,q_j}^{\prime}\\
		\end{array}
	\right.\,,
\end{equation}	
where the points $x_{i,q_j}\,, x_{i,q_j}^{\prime} \in I_{q_j}=[q_j,q_{j+1}]$, and they satisfy $f^q(x_{i,q_j})=x_i$ and $f^q(x_{i,q_j}^{\prime})=x_i^{\prime}$.

\begin{itemize}
\item[(i)]{If $f^q$ does not have extrema in $A_{ij}$ (see Figs.~\ref{fig:figura1b} and \ref{fig:figura1c}), then we let
\begin{equation}
	W_{q_{ij}}=A_{ij}\,.
\end{equation}
}
\item[(ii)]{If $f^q$ has extrema in $A_{ij}$ (see Figs.~\ref{fig:figura1b} and \ref{fig:figura1d}), then we let\\
\begin{equation}
	W_{q_{ij}}= 
		\left\{
			\begin{array}{c}
				(x_{i,q_j},\mathrm{C}{]}\,, \qquad \mbox{if} \quad x_{i,q_j}< C\\
				{[}\mathrm{C},x_{i,q_j})\,, \qquad \mbox{if} \quad x_{i,q_j}> C \\
			\end{array}
		\right.\,.
\end{equation}
}
\end{itemize}

\end{definition}

{\bf Remark:} If we did not take point (ii) into account, then $f^q$ would not be monotonic in $W_{q_{ij}}$. 

\vspace{.3cm}

By construction, all points $x \in W_{q_{ij}}$ reach $I_{P_i}$ in at most $q$ iterations (see Figs.~\ref{fig:figura1b}, \ref{fig:figura1c}, and \ref{fig:figura1d}).  That is, $f^l(W_{q_{ij}})=I_{P_i}$ for $l \leq q$.

\begin{definition}\label{def:WR}
We call
$\displaystyle W_{R_i}= \bigcup_j W_{q_{ij}}$  the {\tt $q$-capture interval of $I_{P_i}$}, as $I_{P_i}$ is captured after at most $q$ iterations. The interval
$\displaystyle W_R=\bigcup_{i} W_{R_i}=\bigcup_{i,j} W_{q_{ij}}$ is then the {\tt $q$-capture interval of the orbit $\{  S_i  \}_{i=1}^p$}.
\end{definition}

Observe that $W_{q_{ij}}$ can be the empty set for some values of $j$.


\section*{Theorems} \label{thms}

Once we know $W_R$, we can calculate the probability that a point picked uniformly at random from phase space is located in $W_R$. We can then calculate the probability that that point reaches a capture interval of $O_p$ in at most $q$ iterations.  We let $\mu (W_R)$ denote the measure of $W_R$, and we have the following theorem.

\begin{theorem}\label{th:Pr}
Let $O_p=\{ S_i \}_{i=1}^p$ be a stable $p$-periodic orbit of the system (\ref{eq-1}). Given an arbitrary point $x \in I$, the probability to reach a capture interval of $O_p$ after at most $q$ iterations is
\begin{equation}
	P_q=\frac{\mu (W_R)}{\mu (I)} \; = \; \frac{\mu (W_R)}{b-a}\,.
\end{equation}	
\end{theorem}

\begin{proof}
From the definition (\ref{def:WR}) of $W_R$, all $x \in W_R$ satisfy $f^l(x) \in I_p$ for $l\leq q$.  There always exist values of $l<q$ such that $f^l(x) \in I_p$ because extrema of $f^l(x)$ that satisfy $l<q$ are necessarily also extrema of $f^q$, and points belonging to the latter set of extrema reach a capture interval of $O_p$ after at most $l$ iterations (see Lemma \ref{lem:extrema} below). Consequently, one reaches $I_p$ from $x \in W_R$ after at most $q$ iterations (and we note that it need not be exactly $q$ iterations). Thus, the probability to reach $I_p$ from an arbitrary point $x \in I$ after at most $q$ iterations (i.e., the probability that $x \in W_R$) is
\begin{equation}
	P_q=\frac{\mu (W_R)}{\mu (I)} \; = \; \frac{\mu (W_R)}{b-a}\,.
\end{equation}	
\end{proof}

\begin{corollary}
With the hypotheses of Theorem \ref{th:Pr}, the probability to reach a capture interval of $O_p$ in exactly $q$ iterations is
\begin{equation*}
	P_q - P_{q-1}\,.
\end{equation*}	
\end{corollary}

This answers the question of how long it takes to reach a capture interval of a $p$-periodic orbit from an arbitrary point. However, we also need to calculate $\mu (W_R)$. To do this, we need to understand the structure of $W_{q_{ij}}$. As the following lemma indicates, some of these subintervals are located where $f^q$ is monotonic and others contain extrema of $f^q$.

\begin{lemma}\label{lem:extrema}
If $f: I \rightarrow I$ is an unimodal $\mathcal{C}^0$ function with a critical point at $C$, then $f^q(x)$ has extrema
\begin{itemize}
\item[(i)]{at points for which $f^{q-1}(x)=C$;}
\item[(ii)]{at the same points at which $f^{q-1}(x)$ has extrema.}
\end{itemize}
\end{lemma}

\begin{proof}

\begin{itemize}
\item[(i)]{For all $x \in I$ such that $f^{q-1}(x)=C$, we know that $f^q(x)= f(f^{q-1}(x))=f(C)$. Therefore, $f^q$ has an extremum because $f$ has an extremum.}
\item[(ii)]{Write $I= J_L \cup \{ \mathrm{C} \} \cup J_R$, where $J_L=(a,\mathrm{C})$ and $J_R=(\mathrm{C},b)$, so $f$ is a monotonic function on the intervals $J_L$ and $J_R$.

(ii.a) If $x \in J_L$ or $x \in J_R$ and the function $f^{q-1}(x)$ has an extremum, then we know that $f^{q-1}(x)$ is a monotonically increasing function on one side of $x$ and a monotonically decreasing function on the other. Consequently, $f^q(x)=f(f^{q-1}(x))$ is the composition of two monotonic functions ($f$ and $f^{q-1}$), both of which are increasing (or decreasing) on one side of $x$. On the other side of $x$, one of them is increasing and the other is decreasing.  Therefore, there is an extremum at $x$.

(ii.b) Otherwise, if $f^{q-1}(x)={C}$, then we see straightforwardly that $f^q$ has an extremum.
}
\end{itemize}
\end{proof}

We have just seen how to determine the locations of extrema of $f^q$.  
We also need to know the values that $f^q$ takes at these extrema.

As we will see below, if the system (\ref{eq-1}) has a stable $p$-periodic orbit  and $q>p$, then the values that $f^q$ takes at its extrema are the same as those that $f^p$ takes at its extrema. This makes it possible to calculate the subintervals $W_{q_{ij}}$ that are associated with extrema of $f^q$ by using $I_{P_{iC}}$ and the derivative of $f$.

\begin{lemma}\label{lem:coordinates}
 Let $O_p$ be a stable $p$-periodic orbit of the system (\ref{eq-1}). The coordinates of the extrema of $f^q$ (where $q>p$) are $(x_{iC},f^{q-i|_p}({C}))$, where $x_{iC}$ denotes the points $x \in I$ such that $f^i(x)={C}$, the index $i$ takes values of $i=0,1,\ldots,q-1$ (where we note that $f^0 \equiv Id$ is the identity map), and $q-i|_p=(q-i)mod \;p$.
\end{lemma}

\begin{proof}
According to Lemma \ref{lem:extrema}, the extrema of $f^q$ are
\begin{itemize}
\item[(i)]{$x \in I$ such that $f^{q-1}(x)={C}$;}
\item[(ii)]{$x \in I$ such that $f^{q-1}(x)$ is an extremum.}
\end{itemize}

It thus follows that the extrema of $f^{q-1}$ are
\begin{itemize}
\item[(iia)]{$x \in I$ such that $f^{q-2}(x)={C}$;}
\item[(iib)]{$x \in I$ such that $f^{q-2}(x)$ is an extremum.}
\end{itemize}

Repeating the process, we obtain that extrema of $f^q$ are located at $x_{iC}$, where $i=0,1,\ldots,q-1$. The value of $f^q$ at $x_{iC}$ is $f^q(x)=f^{q-i}(f^i(x))=f^{q-i}({C})$.

Because $O_p$ is a stable $p$-periodic orbit, there exists one point of $O_p$ near ${C}$ that is repated periodically after $p$ iterations.
Consequently, $\{{C}, f({C}), f^2 ({C}),\ldots,f^{p-1}({C})  \}$ is a periodic sequence and $f^{q-i}({C})=f^{q-i|_p}({C})$.
\end{proof}


\section*{Algorithms}\label{algs}


As we discussed above, Lemmas \ref{lem:extrema} and \ref{lem:coordinates} determine intervals of monotonicity (see definition \ref{def:monotonicity}), and they also make it possible to construct algorithms for calculating $W_{q_{ij}}$.

For these algorithms, we approximate $f^q$ by line segments in the subintervals in which $f^q$ is monotonic.  This approximation is very good unless one is extremely close to an extremum (see Fig.~\ref{fig:figura2}), and this is already the case even for relatively small $q$ (as we will demonstrate below).  
Additionally, recall that $W_{q_{ij}}$ is determined by the intersection points of $f^q$ with line segments. Therefore, once we have approximated $f^q$ by a set of line segments, it is straightforward to calculate those intersection points.

\vspace{1cm}

\begin{algorithm}\label{alg-1}
{\tt (Calculating coordinates for extrema of $f^q$)}  Suppose that we know the coordinates of the extrema of $f^{q-1}$. According to Lemma \ref{lem:extrema}, the extrema of $f^q$ are located at the points
\begin{itemize}
\item[(i)]{$x \in I$ such that $f^{q-1}(x)={C}$ and $f^{q-1}(x)$ is not an extremum;}

\item[(ii)]{$x \in I$ such that $f^{q-1}(x)$ is an extremum.}
\end{itemize}


We know the extrema in (ii) by hypothesis. To find the extrema in (i), we need to calculate the points $x \in I$ that satisfy $f^{q-1}(x)={C}$. Because we know the coordinates of extrema of $f^{q-1}$, we construct the lines that connect two consecutive extrema (see Fig.~\ref{fig:fig_example}). Let $x_{n+1}=a x_n +b$ be the equation for such a line. We solve $a x_n +b = {C}$ to obtain a seed that we can use in any of the many numerous numerical methods for obtaining roots of nonlinear algebraic equations. Observe that $f^q$ is monotonic in the interval in which the line  $x_{n+1}=a x_n +b$ is defined. This circumvents any problem that there might otherwise be in obtaining a good seed to ensure convergence of the root solver. Moreover, we have as many seeds as there are points $x \in I$ that satisfy $f^{q-1}(x)={C}$. Note that we need to construct both the line that connects $(a,f(a))$ with the first extremum of $f^{q-1}$ and the line that connects $(b,f(b))$ with the last extremum of $f^{q-1}$.

To calculate the points $x \in I$ for which $f^{q-1}(x)$ is an extremum, we apply this algorithm recursively, and we note that we know by hypothesis that $f$ has an extrememum at $C$.
 We first build the line segments that connect $(a,f(a))$ with $({C},f({C}))$ and $({C},f({C}))$ with $(b,f(b))$. These two line segments give seeds from which to determine the points $x \in I$ that satisfy $f(x)={C}$.  We thereby obtain the coordinates for the extrema of $f^2$.  We then use the same procedure to obtain the coordinates for extrema of $f^3$, $f^4$, $\ldots$ , $f^q$.
\end{algorithm}

\vspace{1cm}

We will see below that if the system (\ref{eq-1}) has a stable $p$-periodic orbit and $q \gg p$, then the points $x \in I$ with $f^{q-1}(x)={C}$ are given to a very good approximation by the intersection points of two lines.  Moreover, as one can see in Fig.~\ref{fig:figura2}, the value $q=6$ is already large enough to approximate $f^q$ very successfully by a set of line segments when $f$ is the logistic map.

\vspace{1cm}

\begin{algorithm}\label{alg-2}
{\tt (Calculation of $W_{q_{ij}}$ in the system} (\ref{eq-1}){\tt)} Suppose that we know the coordinates of the extrema of $f^q$ (e.g., by computing them using Algorithm \ref{alg-1}). We want to obtain $W_{q_{ij}}$ from the definition (\ref{def:W}), where 
\begin{equation}
	f^q(x_{i,q_j})=x_i\,, \qquad f^q(x_{i,q_j}^{\prime})=x_i^{\prime}
\end{equation}
and the $i$th capture interval is
\begin{equation}
	I_{P_i} = 
		\begin{cases}
			(x_i,x_i^{\prime})\,,\quad \mbox{if} \qquad x_i<x_i^{\prime}\,, \\
			(x_i^{\prime},x_i)\,,\quad \mbox{if} \qquad x_i>x_i^{\prime}\,.
		\end{cases}
\end{equation}

To determine the points $x_{i,q_j}$ and $x_{i,q_j}^{\prime}$, we first approximated them by replacing $f^q$ by line segments that connect consecutive extrema of $f^q$ (i.e., by the same procedure that we use in Algorithm \ref{alg-1} to obtain approximations of points). Using the approximations of $x_{i,q_j}$ and $x_{i,q_j}^{\prime}$, we construct the interval $I_{\mathrm{app}}=(x_{i,q_j},x_{i,q_j}^{\prime})$ and then check if there is an extremum of $f^q$ in $I_{\mathrm{app}}$. (This is trivial because we know the coordinates of the extrema of $f^q$.) We need to consider two cases.
\begin{itemize}
\item[(i)]{The map $f^q$ has no extrema in $I_{\mathrm{app}}$. This is equivalent to case (i) of Algorithm \ref{alg-1}. We use the approximations of $x_{i,q_j}$ and $x_{i,q_j}^{\prime}$ as seeds in a numerical root-finding method.}
\item[(ii)]{The map $f^q$ has extrema in $I_{\mathrm{app}}$. This is equivalent to case (ii) of Algorithm \ref{alg-1}. 
}
\end{itemize}

If there is an extremum of $f^q$ in $I_{\mathrm{app}}$, then that extremum is necessarily one of the extrema given by Lemma \ref{lem:coordinates}: $(x_{iC},f^{q-i|_p}(\mathrm{C}))$. Because $f^i(x_{iC})={C}$ and $f$ is a continuous function, there must exist an interval $I_{iC}$ such that $x_{iC} \in I_{iC}$ and $f^i(I_{iC}) \subset I_{P_{i,C}}$

Taking into account that $x_{iC}$ is known, we construct the sequence 
\begin{equation*}
	S_{iC}=\{x_{i,0},x_{i,1},\ldots,x_{i,i} \equiv \mathrm{C}   \}\,,
\end{equation*}	
where $x_{iC} \equiv x_{i,0}$, $x_{i,k}=f^k(x_{iC})$, and $f^0(x_{iC})=x_{iC}$\\

Let $L_{i,k}$ be the linear map whose graph is the line of slope $f^{\prime}(x_{i,k})$ that intersects the point $x_{i,k}$. If the period $p$ of the orbit is sufficiently large, then we can approximate $f$ near $x_{i,k}$ (where $k=0,1,\ldots,i-1$) by the linear map $L_{i,k}$. Thus, instead of iterating $I_{iC}$ with the map $f$ to obtain $I_{P_{i,C}}$, we iterate $I_{iC}$ with the linear map $L_{i,k}$ that approximates $f$.  That is,
\begin{equation*}
	I_{P_{i,C}} \approx L_{i,i-1} \ldots L_{i,0} (I_{iC})\,.
\end{equation*}
Because each $L_{i,k}$ is a linear map, it is straightforward to compute $L^{-1}_{i,k}$ and hence to compute
\begin{equation*}
	I_{iC} \approx L^{-1}_{i,0} \ldots L^{-1}_{i,i-1} (I_{P_{i,C}})\,.
\end{equation*}
At the end of this section, we will discuss the error that is introduced by this approximation.

The interval $I_{iC}$ that we have just constructed is the interval 
\begin{equation}
	W_{q_{ij}}= 
\left\{
	\begin{array}{c}
		(x_{i,q_j},\mathrm{C}{]}\,, \quad \mbox{if} \qquad \mbox{if} \; x_{i,q_j}<{C}\\
		{[}\mathrm{C},x_{i,q_j})\,, \quad \mbox{if} \qquad \mbox{if} \; x_{i,q_j}>{C}\\
	\end{array}
\right.
\end{equation}
that we seek.
\end{algorithm}

\vspace{1cm}

In Algorithm \ref{alg-1}, we constructed line segments that connect two consecutive extrema of $f^q$. They are located in the intervals $[q_j,q_{j+1})$ and $[q_{j+1},q_{j+2})$, respectively. We now have intervals $W_{q_{i,j}} \subset[q_j,q_{j+1})$ and  $W_{q_{i,j+1}} \subset[q_{j+1},q_{j+2})$ that contain these two consecutive extrema of $f^q$,
 so we construct the line segment that connects the upper endpoint of $W_{q_{i,j}}$ to the lower endpoint of $W_{q_{i,j+1}}$.  (Note that we \emph{do not} connect the two extrema directly via a line segment.)  For $q \gg 1$, this line segment approximates $f^q$ outside of the intervals $W_{q_{i,j}}$ and $W_{q_{i,j+1}}$. See Fig.~\ref{fig:figura2}, which illustrates (for the case when $f$ is the logistic map) that we can approximate $f^6$ by a set of line segments for $q=6$. We can then use these line segments in Algorithm \ref{alg-1}, and we do not need numerical computations to find the intersection points. 

As we discussed previously, we can replace $f^q$ by linear expressions to approximate the intersection points when determining $W_R$ in Algorithms 1 and 2. Replacing $f^q$ by a linear approximation simplifies operations and reduces the amount of calculation. To determine the desired intersection points, we have thereby replaced a numerical method for obtaining roots of nonlinear algebraic equations by an analytical calculation that uses a system of two linear equations.   
We now estimate the error of replacing $f^q$ by lines segments. The line segments that replace the function $f^q$ intersect $f^q$ very close to the unique point of inflection between a pair of consecutive extrema of $f^q$ (see Remark \ref{rem:schw_der} and Fig.~\ref{fig:fig_example}). The Taylor polynomial of degree 3 
of $f^q$ around the inflection point $x_{\mathrm{inf}}$ is 
\begin{align*}
	f^q(x)\simeq f^q(x_{\mathrm{inf}}) + f^{q\prime}(x_{\mathrm{inf}})(x - x_{\mathrm{\mathrm{inf}}}) + \frac {1}{3!} f^{q\prime\prime\prime} (x_{\mathrm{inf}})(x - x_{\mathrm{inf}})^3\,. 
\end{align*}	
Consequently, the error of approximating $f^q$ by the line $f^q(x_{\mathrm{inf}}) + f^{q\prime}(x_{\mathrm{inf}})(x - x_{\mathrm{inf}})$ is
\begin{align}\label{decrease}
	\mbox{Error} = \left|\frac {1}{3!} f^{q\prime\prime\prime} (x_{\mathrm{inf}})(x - x_{\mathrm{inf}})^3\right| \approx \left|\frac {1}{3!} f^{q\prime\prime\prime} (x_{\mathrm{inf}})\left(\frac {b-a}{2^q}\right)^3\right|\,,
\end{align}	
where we have taken into account that there are more than $2^q$ local extrema of $f^q$ in the interval $[a,b]$. The exponential growth of $2^q$ enforces a fast decay in the error. Consequently, using line segments to approximate $f^q$ is an effective procedure with only a small error.




\section*{Numerical Example}\label{numer}

Algorithms \ref{alg-1} and \ref{alg-2} are based on the same procedure: approximate $f^q(x)$ by a line $y(x)= a x + b$ and solve $y(x)=C$ to obtain an approximation of the $f^q(x)= C$ (instead of solving $f^q(x)= C$ directly).  In this section, we consider an example application of Algorithm \ref{alg-1}. 

To obtain the critical points of $f^{q+1}$, we need to calculate the points that satisfy $f^q=C$. Suppose that $q=6$ (and again see Fig.~\ref{fig:figura2} for an illustration of the line-segment approximation with $q = 6$ for the logistic map).  The biggest distance between consecutive extrema occurs near the critical point $C$, so we approximate $f^6$ by a line segment in this region to obtain an upper bound for the error. The extrema are located at  $(4.525 \times 10^{-1},2.414 \times 10^{-3})$ and $(4.787 \times 10^{-1},9.994 \times 10^{-1})$, and they are connected by the line $y \approx38.053  \;x - 16$, from which we obtain the approximation $x_{\mathrm{app}} \approx 0.453$ for the solution of $f^6(x)=C$.  From direct computation, the value of $x$ that satisfies $f^6(x)=C$ is $x \approx 0.465$. The relative error is $E_{\mathrm{rel}}\approx 2.58\%$, and this is the largest error in this example from all of the approximating lines segments.  As we showed in equation (\ref{decrease}), the error decreases exponentially. Hence, when we approximate $f^{6+m}$ using line segments, the relative error will be bounded above by $E_{\mathrm{rel}} \approx 2.58/2^m \; \%$.  One can observe this decrease in error in Fig.~\ref{fig:fig_example}, in which we plot both $f^6$ and $f^{10}$ for the logistic map and the same parameter value $r$. Observe that several extrema of $f^{10}$ lie between onsecutive extrema of $f^6$, so using the line-segment approximation in $f^{10}$ induces a much smaller error than using it in $f^6$.


\section*{Conclusions and Discussion}\label{conc}

When studying dynamical systems, it is important to consider not only whether one converges to periodic orbits but also how long it takes to do so.  We show how to do this explicitly in one-dimensional discrete dynamical systems governed by unimodal functions.
We obtain theoretical results on this convergence and develop practical algorithms to exploit them.  These algorithms are both fast and simple, as they are linear procedures. One can also apply our results to multimodal one-dimensional maps by separately examining regions of parameter space near each local extremum.

Although we have focused on periodic dynamics, the ideas that we have illustrated in this paper can also be helpful for trying to understand the dynamics of chaotic systems.  Two important properties of a chaotic attractor are that (i) its skeleton can be constructed (via a ``cycle expansion") by considering a set of infinitely many unstable periodic orbits; and (ii) small neighborhoods of the unstable orbits that constitute the skeleton are visited ergodically by dynamics that traverse the attractor \cite{chaosbook}.  In Refs.~\cite{smelch1,smelch2}, Schmelcher and Diakonos developed a method to detect unstable periodic orbits of chaotic dynamical systems. They transformed the unstable periodic orbits into stable ones by using a universal set of linear transformations.  One could use the results of the present paper after applying such transformations.  Moreover, the smallest-period unstable periodic orbits tend to be the most important orbits for an attractor's skeleton \cite{chaosbook}, so our results should provide a practical tool that can be used to help gain insights on chaotic dynamics.

Once unstable orbits has been transformed into stable ones we can use results of this paper to answer the above question.


\section*{Acknowledgements}

We thank Erik Bollt, Takashi Nishikawa, Adilson Motter, Daniel Rodr\'iguez, and Marc Timme for helpful comments.


\bibliography{jesus}


\section*{Figure Legends}


\begin{figure}[!ht] 
  \centering
  \includegraphics[width=0.8\textwidth]{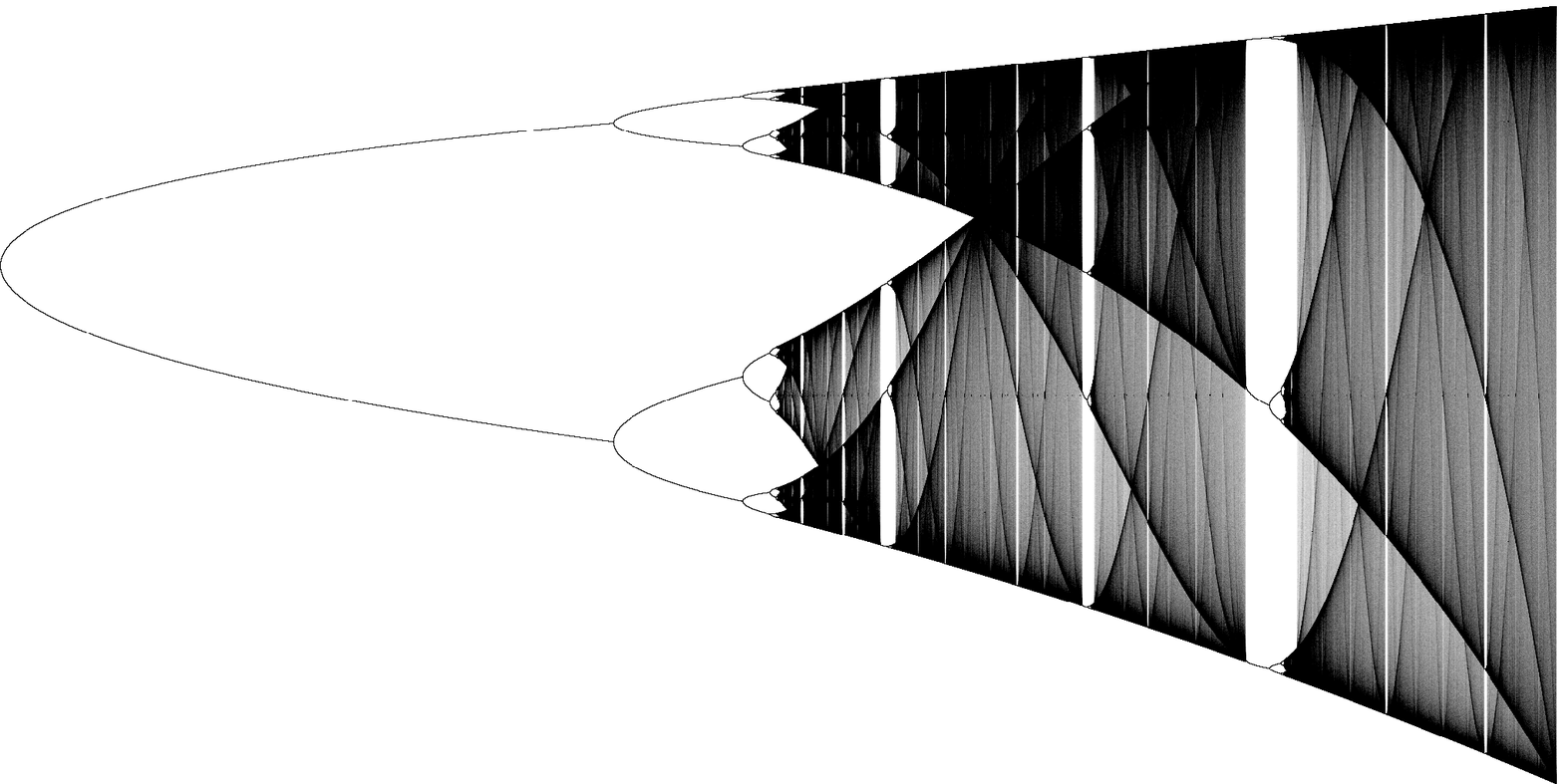}
  \caption{{\bf Bifurcation diagram of a unimodal map with a negative Schwarzian derivative.} There is a period-doubling cascade on the left, and there are also period-doubling cascades inside several windows (the broad, clear bands) of periodic behavior. Saddle-node orbits arise at the onset of such windows in the chaotic area.
  }
  \label{fig:diagrama}
\end{figure}

\begin{figure}[!ht] 
  \centering
  \includegraphics[bb=0 0 397 278,width=5.67in,height=3.97in,keepaspectratio]{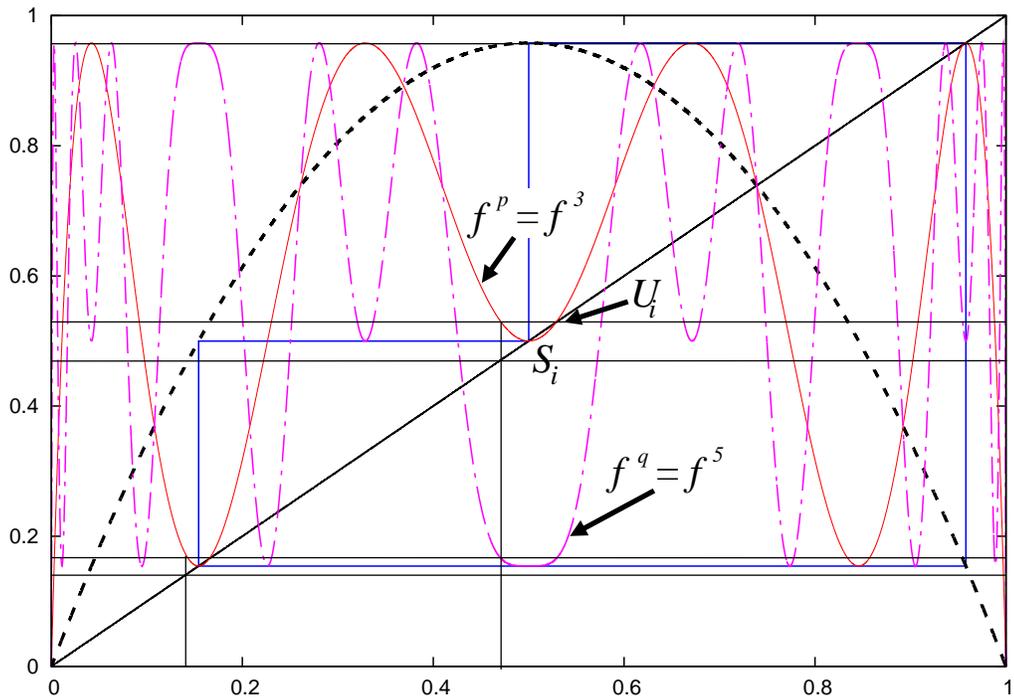}
  \caption{{\bf Geometric calculation of the three subintervals $W_{q_{ij}}$ corresponding to a 3-periodic orbit (in blue). See Fig.~\ref{fig:figura1a} for a better view of the orbit.}   
            These subintervals are determined by the three pairs of black, horizontal, parallel line segments that intersect $f^q$, $U_ {i}$, and $f^p$.
               (We only indicate one $U_i$ in the figure.)  One needs to take into account the intersection points of $f^q$ with all 6 parallel line segments. See Figs.~\ref{fig:figura1c} and \ref{fig:figura1d} for more detail.  The plot in this figure uses the logistic map.  The blue orbit is a period-3 supercycle and $r \approx 3.83187405528331556841$. 
}
  \label{fig:figura1b}
\end{figure}

\begin{figure}[!ht] 
  \centering
  \includegraphics[bb=0 0 397 278,width=5.67in,height=3.97in,keepaspectratio]{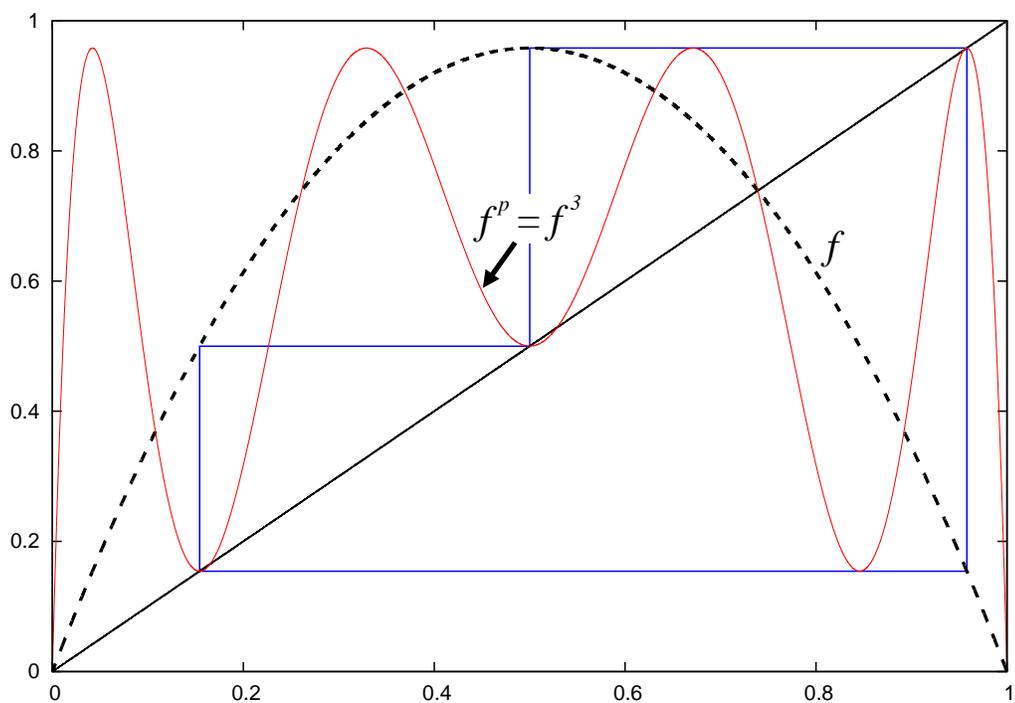}
 \caption{{\bf The 3-periodic orbit $f^p=f^3$ from Fig.~\ref{fig:figura1b}.}}
  \label{fig:figura1a}
\end{figure}

\begin{figure}[!ht] 
  \centering
  \includegraphics[bb=0 0 397 278,width=5.67in,height=3.97in,keepaspectratio]{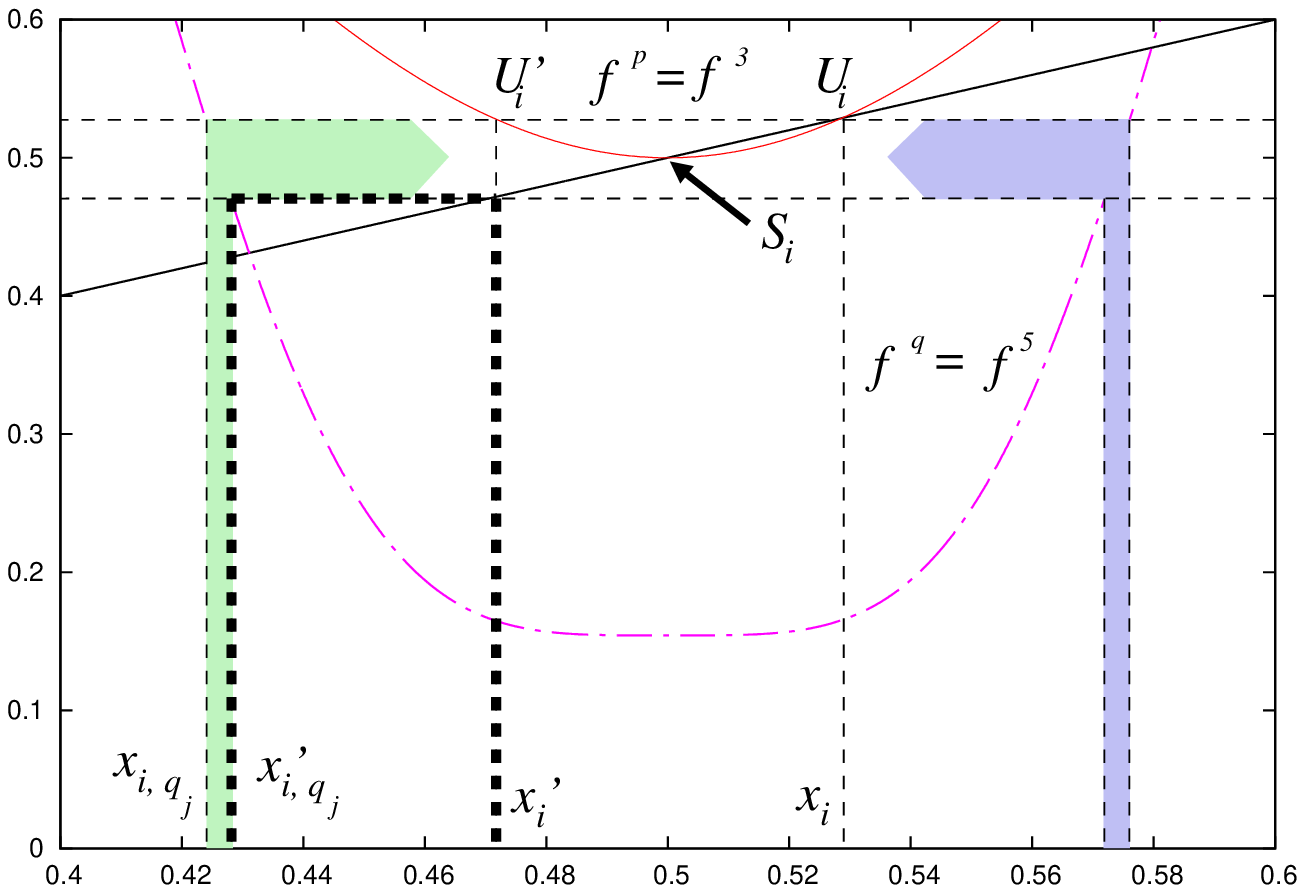}
  \caption{{\bf Magnification of Fig.~\ref{fig:figura1b}.} We show the interval $W_{q_{ij}}$, in which $f^q$ does not have any extrema in the region between the horizontal parallel lines. The horizontal line that crosses $U_i$ and intersects with $f^p$ determines $U_i^{\prime}$. The vertical lines that intersect $U_i$ and $U_i^{\prime}$ determine $x_i$ and $x_{i}^{\prime}$, respectively. We obtain locations for the points $x_{i,q_j}$ and $x_{i,q_j}^{\prime}$ because their images under the map $f^q$ are $x_i$ and $x_{i}^{\prime}$, respectively.  We thereby construct the subinterval $W_{q_{ij}}$. 
  We depict the mapping of the subinterval $W_{q_{ij}}$ using a filled green arrow the mapping of another subinterval using the filled blue arrow.
           }
  \label{fig:figura1c}
\end{figure}

\begin{figure}[!ht] 
  \centering
  \includegraphics[bb=0 0 394 276,width=5.67in,height=3.97in,keepaspectratio]{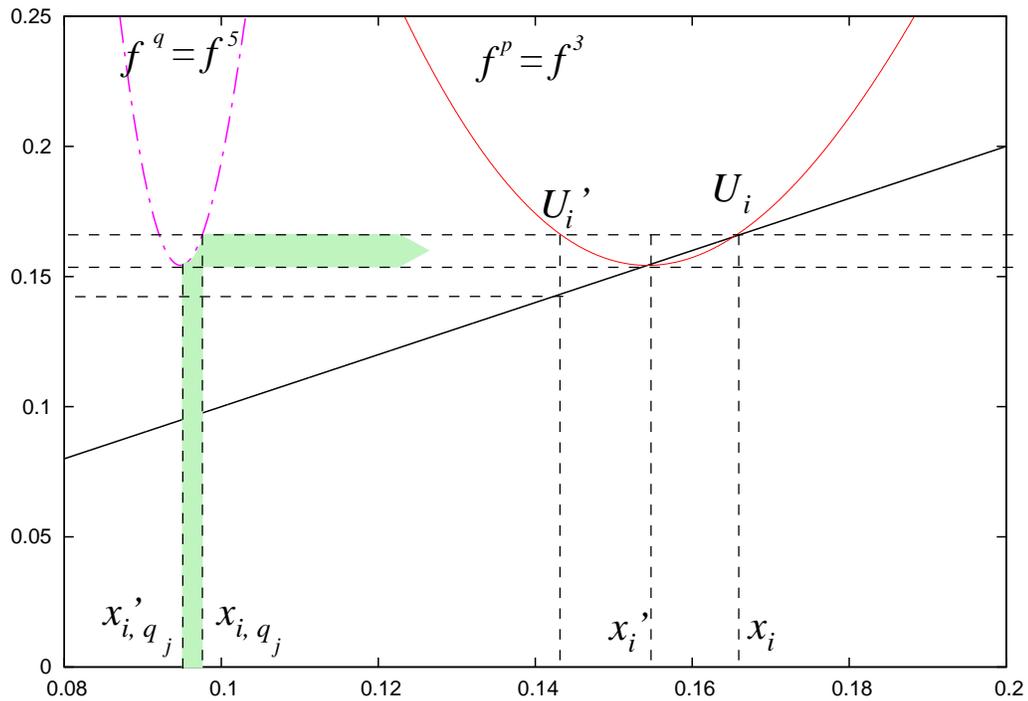}
  \caption{{\bf Another magnification of Fig.~\ref{fig:figura1b}.}  We show the interval $W_{q_{ij}}$,  in which $f^q$ has an extremum in the region between the horizontal parallel lines. The horizontal line that crosses $U_i$ and intersects $f^p$ determines $U_i^{\prime}$. The vertical lines that intersect $U_i$ and $U_i^{\prime}$ determine the points $x_i$ and $x_{i}^{\prime}$, respectively. We obtain the locations for the points $x_{i,q_j}$ and $x_{i,q_j}^{\prime}$ because their images under the map $f^q$ are $x_i$ and $x_{i}^{\prime}$, respectively.  We thereby construct the subinterval $W_{q_{ij}}$. We depict the mapping of the subinterval $W_{q_{ij}}$ using a filled green arrow.
  }
  \label{fig:figura1d}
\end{figure}

\begin{figure}[!ht] 
  \centering
  \includegraphics[width=5.67in,height=3.97in,keepaspectratio]{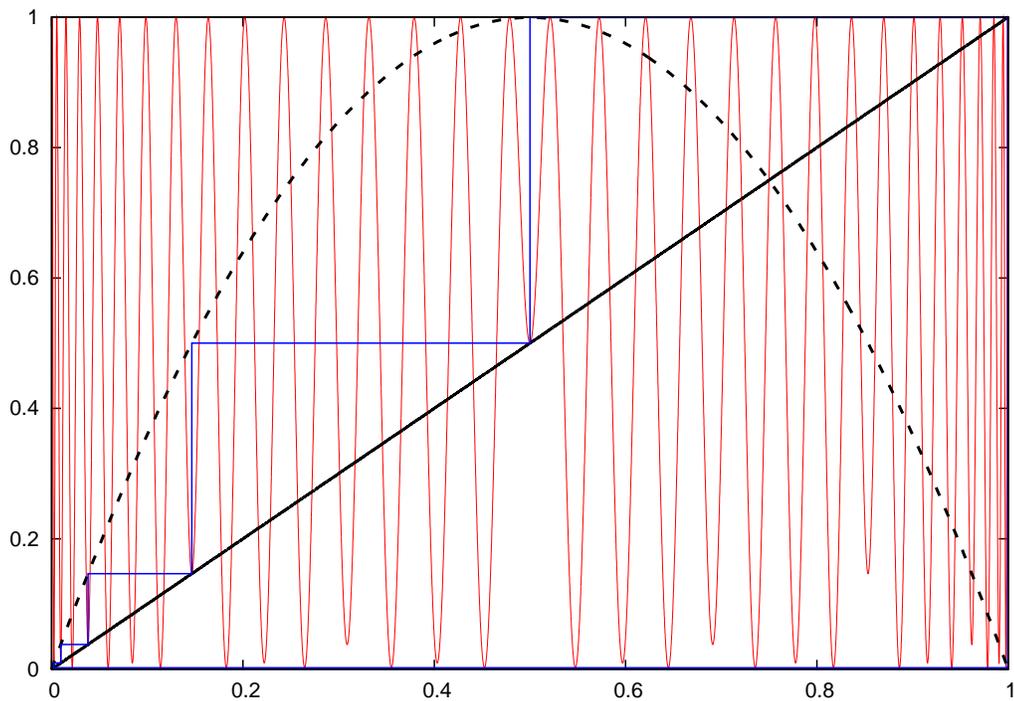}
  \caption{{\bf Outside of the intervals $W_{q_{i,j}}$, we approximate the map $f^q$ using line segments.} A line segments connects the upper endpoint of the interval $W_{q_{i,j}}$ to the lower endpoint of $W_{q_{i,j+1}}$.   The map $f^6$ is very well approximated using line segments as long as one is not too close to an extremum.  We again use the logistic map to illustrate our procedure.  The blue curve is a period-6 supercycle and $r \approx 3.99758311825456726610$. See Fig.~\ref{fig:fig_example} for a magnification of this figure.
  }
  \label{fig:figura2}
\end{figure}

\begin{figure}[!ht] 
  \centering
  \includegraphics[width=0.8\textwidth]{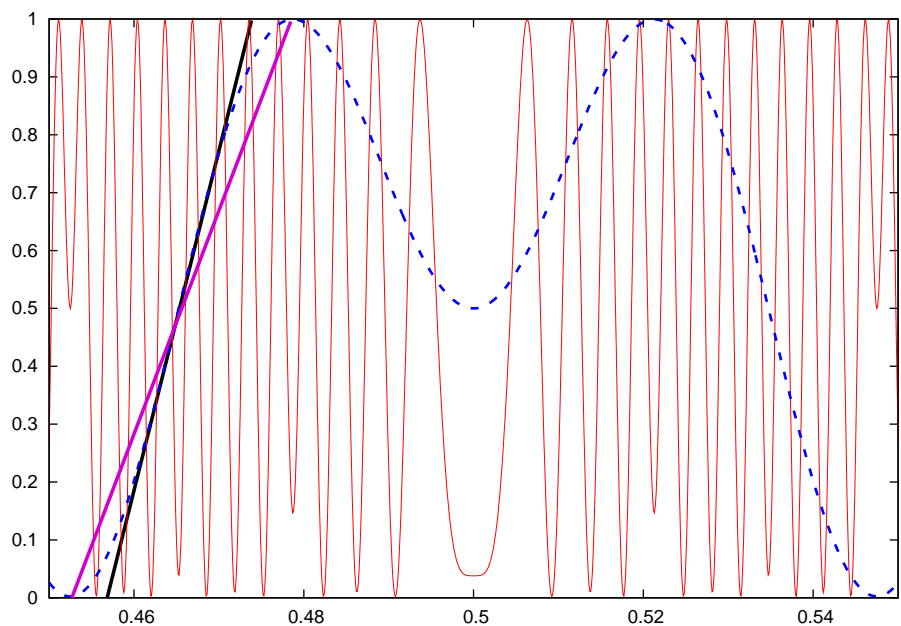} 
  \caption{{\bf Graphs of $f^6$ (blue) and $f^{10}$ (red) for the same value of the parameter $r$ (when $f$ is the logistic map) as in Fig.~\ref{fig:figura2}.} The dark pink line joins two consecutive extrema of $f^6$, and the black line is the tangent line that crosses through the inflection point. Both lines are approximations to $f^q$. As expected, the approximation is better for the larger value of $q$.
  }
  \label{fig:fig_example}
\end{figure}



\end{document}